\documentclass[12pt,reqno]{article}

\usepackage[usenames]{color}
\usepackage{amssymb}
\usepackage{amsmath}
\usepackage{amsthm}
\usepackage{amsfonts}
\usepackage{enumerate}
\usepackage{url}

\usepackage[colorlinks=true,
linkcolor=webgreen,
filecolor=webbrown,
citecolor=webgreen]{hyperref}
\definecolor{webgreen}{rgb}{0,.5,0}
\definecolor{webbrown}{rgb}{.6,0,0}

\makeatletter
\g@addto@macro\bfseries{\boldmath}
\makeatother

\newcommand{\seqnum}[1]{\href{https://oeis.org/#1}{\rm \underline{#1}}}

\newcommand{\ord}{\mathrm{ord}_p}
\newcommand{\inc}{\mathrm{inc}_p}
\newcommand{\midd}{\; || \;}
\newcommand{\pfloor}[1]{\lfloor #1 \rfloor_p}

\theoremstyle{plain}
\newtheorem{theorem}{Theorem}
\newtheorem{corollary}[theorem]{Corollary}
\newtheorem{lemma}[theorem]{Lemma}

\theoremstyle{definition}
\newtheorem{definition}[theorem]{Definition}
\newtheorem{example}[theorem]{Example}
\newtheorem{conjecture}[theorem]{Conjecture}
\newtheorem{remark}[theorem]{Remark}

\begin{document}

\begin{center}
\vskip 1cm{\LARGE\bf 
The Arithmetic Partial Derivative}
\vskip 1cm
\large
Brad Emmons and Xiao Xiao\\
Department of Mathematics\\
Utica University \\
Utica, NY 13502 \\
U.S.A. \\
\href{mailto:bemmons@utica.edu}{\tt bemmons@utica.edu} \\
\href{mailto:xixiao@utica.edu}{\tt xixiao@utica.edu}
\end{center}

\vskip 0.2in

\begin{abstract}
The arithmetic partial derivative (with respect to a prime $p$) is a function from the set of integers that sends $p$ to 1 and satisfies the Leibniz rule. In this paper, we prove that the $p$-adic valuation of the sequence of higher order partial derivatives is eventually periodic. We also prove a criterion to determine when an integer has integral anti-partial derivatives. As an application, we show that there are infinitely many integers with exactly $n$ integral anti-partial derivatives for any nonnegative integer $n$.
\end{abstract}

\section{Introduction}
Let $p$ be a prime number. We recall the $p$-adic valuation $\ord: \mathbb{Z} \to \mathbb{N} \cup \{+\infty\}$. For any $x \in \mathbb{Z}$, we have
\begin{displaymath}
    \ord(x) = \begin{cases}
        \max\{\ell \in \mathbb{N} : p^{\ell} \mid x\}, & \text{if $x \neq 0$};\\
        +\infty, & \text{if $x=0$.}
        \end{cases}
\end{displaymath}
In other words, if $x \neq 0$, then $\ord(x)$ is the largest integer $\ell$ such that $p^{\ell} \mid x$. We sometimes denote this by $p^{\ell} \midd x$. 

The arithmetic derivative of a nonzero integer $x$ is a function $D: \mathbb{Z} \backslash \{0\} \to \mathbb{Z}$ defined as 
\[ D(x) = x \sum_{p | x}  \frac{\ord(x)}{p}.\]
If $x = p^k$ is a prime power, then 
\[D(p^k) = p^k \cdot \frac{k}{p} = kp^{k-1}.\]
This shows that $D$ mimics the power rule for differentiable functions on prime powers. One can check that $D$ satisfies the Leibniz rule $D(xy) = xD(y) + yD(x)$, and sends every prime to $1$. By the Leibniz rule, we are able to extend $D$ uniquely to $\mathbb{Z}$ by defining $D(0) = 0$. 

Ufnarovski and {\AA}hlander \cite{Ufnarovski} made several conjectures related to this arithmetic derivative, two of which we will investigate here.  The first conjecture relates to the dynamical system by considering higher order arithmetic derivatives of an integer. We denote by $D^i(x)$ the $i$-th arithmetic derivative of $x$ for $i \geq 1$. 

We note that the first conjecture that we would like to address is not correct as stated. The original conjecture \cite[Conjecture 2]{Ufnarovski} states that for any positive integer $x$, exactly one of the following could happen: either $D^i(x) = 0$ for sufficiently large $i$, $\displaystyle \lim_{i \to +\infty} D^i(x) = +\infty$, or $x = p^p$ for some prime $p$. The fact that $D(1647082) = 823543 = 7^7$ means that it is not quite true.  We modify it slightly here to the following conjecture. 

\begin{conjecture} \label{conjecture:UA3}
For any positive integer $x$, exactly one of the following could happen: either $D^i(x) = 0$ or $p^p$ for some prime $p$ for sufficiently large $i$, or $\displaystyle \lim_{i \to +\infty} D^i(x) = +\infty$.
\end{conjecture}

The second conjecture relates to the number of anti-derivatives of an integer. An integer $y$ is said to be an (integral) anti-derivative of $x$ if $D(y) = x$. Unlike anti-derivatives of functions, for any positive integer $x > 1$, $x$ has only finitely many anti-derivatives \cite[Corollary 3]{Ufnarovski}. We denote by $i(x)$ the number of anti-derivatives of an integer $x$.

\begin{conjecture} \label{conjecture:UA4} 
For each nonnegative integer $n$, the number of solutions to the equation $i(x) = n$ is infinite. 
\end{conjecture} 

Part of what makes investigating the function $D(x)$ challenging is that the derivative of an integer $x$ will often have prime factors that are not factors of $x$ itself.  For instance, if $x = 72 = 2^3 \, 3^2$, then $D(x) = 156 = 2^2 \, 3^1 \, 13$, and we see that while the multiplicities of the prime factors $2$ and $3$ each decreased by $1$, the derivative $D(x)$ has a factor of $13$.  In particular, when investigating the dynamical system $D^i(x)$, we would have to keep track of new prime factors that may arise following each successive differentiation. If $x$ has more than one prime factor, predicting new prime factors of $D(x)$ relies on the ability of predicting prime factors of $a+b$ when knowing the prime factors of $a$ and $b$. A difficult conjecture that tries to tackle a question in this direction is the famous $abc$ conjecture. There is still ongoing debate in the mathematics community about whether $abc$ conjecture has been proved \cite{abc}. Loosely speaking, the $abc$ conjecture implies that when two numbers $a$ and $b$ are divisible by large powers of small primes, then $a + b$ tends to be divisible by small powers of large primes. Because of difficulties like these, we start by investigating these conjectures in the case of arithmetic partial derivatives with respect to a prime $p$, which was formally introduced by Kovi\v{c} \cite{Kovic}.

For any nonzero $x \in \mathbb{Z}$, the (arithmetic) partial derivative (with respect to $p$) is defined as
\[ D_p(x) = \frac{x}{p} \cdot \ord(x).\] 
One can check that $D_p$ satisfies the Leibniz rule $D_p(xy) = D_p(x)y + xD_p(y)$.  And to extend the definition of $D_p$ to $0$, we observe that for any nonzero $x \in \mathbb{Z}$
\[D_p(0) = D_p(0 \cdot x) = D_p(0) \cdot x + 0 \cdot D_p(x) = D_p(0) \cdot x.\]
This implies that $D_p(0) = 0$. If $p$ does not divide $x$, then $D_p(x) = 0$. For every integer $x$, the arithmetic derivative of $x$ is a finite sum of arithmetic partial derivative of $x$, that is,
\[D(x) = \sum_{p \mid x} D_p(x).\]

The main advantage that we have in exploring the dynamical system $D_p^i(x)$ over $D^i(x)$ is that now we do not need to keep track of any new primes that may arise when we differentiate.  This is because if $x = ap^{\ell}$, where $\ell = \ord(x)$, then $D_p(x) = aD_p(p^\ell)$, and so upon repeated differentiation, we only need to keep track of $\mbox{ord}_p(x)$ and its partial derivatives with respect to $p$. Any new prime factors that may arise when differentiating get absorbed into the coefficient.   

Because we do not have to deal with new primes occurring in our dynamical system, we are able to prove a stronger version of Conjecture \ref{conjecture:UA3} for the arithmetic partial derivative.  Theorem \ref{theorem:main1} provides an exact prescription of the $\ord$ sequence of $x$ when $\ord(x) \geq p$, and as a corollary we prove Theorem \ref{theorem:main2}, which is analogous to Conjecture \ref{conjecture:UA3}; see Section \ref{section:ord} for the definition of $\ord$ sequence.  We go on to prove in Theorem \ref{theorem:main_reverse} that given any possible $\ord$ sequence as prescribed in Theorem \ref{theorem:main1}, we are able to construct an integer $x$ with this $\ord$ sequence.  

For each integer $x$, we say that an integer $y$ is an (integral) anti-partial derivative of $x$ if $D_p(y) = x$. Let $i_p(x)$ be the number of anti-partial derivatives of $x$. Theorem \ref{theorem:infinite_derivative} and Example \ref{example:no_antipd} together show that for each nonnegative integer $n$, the number of solutions to the equation $i_p(x) = n$ is infinite, which is analogous to Conjecture \ref{conjecture:UA4}. We think the natural next step is to study analogous conjectures for the arithmetic subderivative $D_S$ defined by Haukkanen, Merikoski, and Tossavainen \cite{HMT2} in the following way:
\[ D_S(x) := x \sum_{p \in S}  \frac{\ord(x)}{p} = \sum_{p \in S}D_p(x), \]
where $S$ is a finite set of primes.

The definition of arithmetic derivative and arithmetic partial derivative can be generalized to rational numbers \cite{Ufnarovski, HMT1}. Haukkanen, Merikoski, and Tossavainen \cite[Conjecture 27]{HMT1} conjecture that for any nonzero rational number $x$, there are at most four rational anti-partial derivatives if $p=2$, and at most two rational anti-partial derivatives if $p> 2$. As a result of Theorem \ref{theorem:infinite_derivative}, we prove that this conjecture is false; see Remark \ref{remark:rational}. This also confirms Pandey and Saxena's earlier observations \cite{PS} that \cite[Conjecture 27]{HMT1} might be false.

\section{Period of the $\ord$ sequence} \label{section:ord}

We fix a prime $p$. For any $x \in \mathbb{Z}$, we refer to the sequence $x, D_p(x), D_p^2(x), \ldots $ as the $D_p$ sequence of $x$, and the sequence $\ord (x), \ord (D_p(x)), \ord (D_p^2(x)), \ldots $ as the $\ord$ sequence of $x$. In this section, we will study properties of $\ord$ sequences and show that they are eventually periodic. For the sake of completeness, we first give the definition of eventually periodic.

\begin{definition}
    Given a sequence $x_1, x_2, \ldots$, we say that it is \emph{eventually periodic} if there exist integers $m, L \geq 1$ such that for all $0 \leq j \leq L-1$ and $i \geq 1$, we have $x_{m+j+iL} = x_{m+j}$. If a sequence is eventually periodic with $m = 1$, then we say the sequence is \emph{periodic}. The smallest possible $L$ is called the \emph{period} of the sequence.
\end{definition}

If the $D_p$ sequence of $x$ is periodic (resp., eventually periodic), then the $\ord$ sequence of $x$ is periodic (resp., eventually periodic), but not  vice versa.

\begin{lemma} \label{lemma:basic1}
Let $x \in \mathbb{Z}$. If $x = 0$, then the $D_p$ sequence of $x$ (and thus the $\ord$ sequence of $x$) is periodic with period 1. If $\ord(x) < p$, then the $D_p$ sequence of $x$ (and thus the $\ord$ sequence of $x$) is eventually periodic with period 1.
\end{lemma}

\begin{proof}
    If $x = 0$, then $D_p(x) = 0$ and thus $D_p^i(x) = 0$ for all $i \geq 0$. Therefore the $D_p$ sequence of $x$ is periodic of period $1$.
    
    Suppose $x \neq 0$. We write $x = ap^{\ell}$ where $p \nmid a$ and $\ord(x) = \ell < p$. We can compute that $D_p^{\ell}(x) = a\ell !$. Since $0 \leq \ell < p$, we know that $p \nmid a\ell!$. This implies that $D_p^{\ell+1}(x) = 0$ and thus $D_p^{i}(x) = 0$ for all $i \geq \ell+1$. Therefore the $D_p$ sequence of $x$ is eventually periodic of period $1$.
\end{proof}

Suppose $\ord(D_p^i(x)) \neq +\infty$ for all $i \geq 1$. To understand how each term in $\ord$ sequence changes, we study the sequence of increments (of consecutive terms) of the $\ord$ sequence 
\[\ord(D_p(x)) - \ord(x), \; \ord(D_p^2(x)) - \ord(D_p(x)), \; \ord(D_p^3(x)) - \ord(D_p^2(x)), \; \ldots\]
and we call this the $\inc$ sequence of $x$.  Recall that for any integer $x \in \mathbb{Z}$, $x \bmod p$ is the unique integer between 0 and $p-1$ that is congruent to $x$ modulo $p$. We denote $\pfloor{x} := x - (x \bmod p)$. In other words, $\pfloor{x}$ is the largest multiple of $p$ that is less than or equal to $x$.

We are ready to handle the case $\ord (x) \geq p$.  

\begin{lemma} \label{lemma:basic2}
    Let $x = ap^{bp^k}$ be a nonzero integer with $p \nmid ab$ and $k \geq 1$. Define $k' := (k-1 \bmod p) + 1$. The first $k'$ terms of the $\inc$ sequence of $x$ are
    \[(k-1, \underbrace{-1, -1, \ldots, -1}_{(k-1 \bmod p) \text{ copies}}).\]
\end{lemma}

\begin{proof}
    Since $D_p(x) = abp^{bp^k+k-1}$ and $p \nmid ab$, we know $\ord(D_p(x)) = bp^k+k-1$. If $k' \geq 2$, then $D_p^2(x) = ab(bp^k+k-1)p^{bp^k+k-2}$. Since $p \nmid ab(bp^k+k-1)$, we know $\ord(D^2_p(x)) = bp^k+k-2$. Following this pattern, we know  that when $1 \leq i \leq k'$, the $(i+1)$-th term of the $D_p$ sequence is
    \[ab \prod_{j=1}^{i-1} (bp^k+k-j) p^{bp^k+k-i}\]
    because $p \nmid ab \prod_{j=1}^{i-1} (bp^k+k-j)$ for any $1 \leq i \leq k'$. The first $k'+1$ terms of the $\ord$ sequence are
    \[(bp^k, bp^k+k-1, bp^k+k-2, \ldots, bp^k+k-k').\]
    Therefore the first $k'$ terms of the $\inc$ sequence of $x$ are
    \[(k-1, \underbrace{-1, -1, \ldots, -1}_{(k-1 \bmod p) \text{ copies}}). \qedhere\]
\end{proof}

\begin{corollary} \label{corollary:basecase}
Let $x = ap^{bp^k}$ be a nonzero integer with $p \nmid ab$ and $1 \leq k \leq p$. The $\ord$ sequence and the $\inc$ sequence of $x$ is periodic of period $k$. 
\end{corollary}

\begin{proof}
    If $1 \leq k \leq p$, then $k' = (k-1 \bmod p) + 1 = k-1+1= k$. The first $k+1$ terms of the $\ord$ sequence are
    \[(bp^k, bp^k+k-1, bp^k+k-2, \ldots, bp^k).\] It is now clear that the $\ord$ sequence and the $\inc$ sequence of $x$ is periodic of period $k$.
\end{proof}

\begin{definition}
For any integer $k \geq 1$, we call the following sequence
\[\mathcal{S}_k := (k-1, \underbrace{-1, -1, \ldots, -1}_{(k-1 \bmod p) \text{ copies}})\]
the \emph{$k$-segment}.
\end{definition}

For any integer $\ell \geq p$, we define $\ell_1 := \ord(\pfloor{\ell})$ and for $i \geq 2$
\begin{equation} \label{eqn:steps}  
\ell_{i} := \begin{cases}
\ord(\pfloor{\ell_{i-1}-1}), & \text{if $\ell_{i-1} < +\infty$};\\
+\infty, & \text{if $\ell_{i-1} = +\infty$}.
\end{cases}
\end{equation} 
It is clear that if $\ell_i < +\infty$, then $\ell_{i+1} < \log_p(\ell_i)$. For any integer $\ell \geq p$, there exists a unique positive integer $N = N(\ell)$ such that $1 \leq \ell_N \leq p$, and $\ell_i = +\infty$ for all $i> N$. 

\begin{theorem} \label{theorem:main1}
    Let $x = ap^{\ell}$ be a nonzero integer with $\ell \geq p$ and $p \nmid a$.The $\inc$ sequence of $x$ is of the form 
    \[(\underbrace{-1, -1, \ldots, -1}_{(\ell \bmod p) \text{ copies}}, \mathcal{S}_{\ell_1}, \mathcal{S}_{\ell_2}, \mathcal{S}_{\ell_3}, \ldots, \mathcal{S}_{\ell_N}, \mathcal{S}_{\ell_N}, \mathcal{S}_{\ell_N}, \ldots).\]
    As a result, the $\ord$ sequence and the $\inc$ sequence of $x$ are eventually periodic of period $\ell_{N}$.
\end{theorem}

\begin{proof}
Let $\ell'_0 = \ell \bmod p$. Then the first $\ell'_0+1$ terms of the $D_p$ sequence of $x$ are
\[(ap^{\ell}, a\ell p^{\ell-1}, a\ell(\ell-1)p^{\ell-2}, \ldots, a \prod_{i=1}^{\ell'_0} (\ell - i) p^{\ell - \ell'_0}).\]
The first $\ell'_0+1$ terms of the $\ord$ sequence of $x$ are
\[(\ell, \ell-1, \ell-2, \ldots, \ell - \ell'_0).\]
Hence the first $\ell'_0$ terms of the $\inc$ sequence of $x$ are
\[(\underbrace{-1, -1, \ldots, -1}_{\ell'_0 \text{ copies}}).\]
Note that $\ell - \ell'_0 = \pfloor{\ell}$ and $\ell_1 = \ord(\pfloor{\ell})$, we can write $D_p^{\ell'_0}(x) = a_0 p^{b_0p^{\ell_1}}$ with $p \nmid a_0b_0$. We can compute $D_p^{\ell_0'+1}(x) = a_0b_0p^{b_0p^{\ell_1}+\ell_1-1}$. Note that $b_0p^{\ell_1} + \ell_1 -1 \bmod p = \ell_1 -1 \bmod p$. 

Let $\ell'_1 := (\ell_1 -1) \bmod p$. The $(\ell'_0+1)$-th term to the $(\ell'_0+\ell'_1+2)$-th term of the $D_p$ sequence of $x$ are
\[(a_0p^{b_0p^{\ell_1}}, a_0b_0p^{b_0p^{\ell_1}+\ell_1-1}, \ldots, a_0b_0\prod_{i=1}^{\ell'_1}(b_0p^{\ell_1}+\ell_1-i)p^{b_0p^{\ell_1}+\ell_1-1-\ell_1'}).  \]
The $(\ell'_0+1)$-th term to the $(\ell'_0+\ell'_1+2)$-th term of the $\ord$ sequence of $x$ are
\[(b_0p^{\ell_1}, b_0p^{\ell_1}+\ell_1 -1, \ldots, b_0p^{\ell_1}+\ell_1-1-\ell_1').\]
The next $\ell_1'+1$ terms of the $\inc$ sequence of $x$ is the $\ell_1$-segment
\[(\ell_1 - 1, \underbrace{-1, -1, \ldots, -1}_{\ell'_1 \text{ copies}}) = \mathcal{S}_{\ell_1}.\]

As $\ell_1 - 1 - \ell_1'  = \pfloor{\ell_1 -1 }$ and $\ell_2 = \ord(\pfloor{\ell_1-1})$, we can write $D_p^{\ell_0'+\ell_1'+1}(x) = a_1p^{b_1p^{\ell_2}}$ with $p \nmid a_1b_1$. There exists a positive integer $N$ such that $1 \leq \ell_N \leq p$. Using induction, we know that the initial terms of the $\inc$ sequence of $x$ is of the form
\[(\underbrace{-1, -1, \ldots, -1}_{(\ell \bmod p) \text{ copies}}, \mathcal{S}_{\ell_1}, \mathcal{S}_{\ell_2}, \mathcal{S}_{\ell_3}, \ldots, \mathcal{S}_{\ell_N}).\]
Corollary \ref{corollary:basecase} implies that if $a_{N-1}p^{b_{N-1}p^{\ell_N}}$, where $p \nmid a_{N-1}b_{N-1}$, is a term in the $D_p$ sequence of $x$, then $\mathcal{S}_{\ell_N}$ will appear repeatedly in the $\inc$ sequence of $x$. This completes the proof of the theorem.
\end{proof}

We summarize our results of Lemma \ref{lemma:basic1} and Theorem \ref{theorem:main1} into the following theorem.

\begin{theorem}
The $\ord$ sequence of any integer $x$ is eventually periodic of period at most $p$.
\end{theorem}

As an application of Theorem \ref{theorem:main1}, we can prove the following theorem that has a similar structure as Conjecture \ref{conjecture:UA3}.

\begin{theorem} \label{theorem:main2}  
    For any $x \in \mathbb{Z}$, exactly one of the following three cases will happen:
    \begin{enumerate}[\normalfont(i)]
        \item $D_p^{i}(x) = 0$ for sufficiently large $i$;
        \item $D_p^i(x) = ap^p$ where $\gcd(a,p) = 1$ for sufficiently large $i$;
        \item $\displaystyle \lim_{i \to +\infty} D_p^i(x) = \pm \infty$.
    \end{enumerate}
\end{theorem}
\begin{proof}
Suppose $x = 0$ or $0 \leq \ord(x) < p$. Lemma \ref{lemma:basic1} implies that $D_p^{i}(x) = 0$ for sufficiently large $i$.

Suppose $\ord(x) \geq p$. Theorem \ref{theorem:main1} implies that the $\ord$ sequence of $x$ is eventually periodic of period $\ell_{N}$. Let $x_1 = ap^{bp^k}$ be the first term of the periodic cycle where $1 \leq k \leq p$ and $p \nmid ab$. If $b = k = 1$, then $D_p(x_1) = D_p(ap^p) = ap^p = x_1$. We know that $D_p^i(x) = ap^p$ for sufficiently large $i$. Suppose $b >1$ or $k > 1$. If $x_1>0$, then
\[x_1 < D_p(x_1) < D_p^2(x_1) < \cdots < D_p^{\ell_{N_0}-1}(x_1) < D_p^{\ell_{N_0}}(x_1) = \alpha x_1,\]
for some $\alpha > 1$. Therefore \[\lim_{i \to +\infty} D_p^i(x) = \lim_{i \to +\infty} \alpha^ix_1 = +\infty.\]
Similarly, if $x_1<0$, then
\[x_1 > D_p(x_1) > D_p^2(x_1) > \cdots > D_p^{\ell_{N_0}-1}(x_1) > D_p^{\ell_{N_0}}(x_1) = \alpha x_1,\]
for some $\alpha > 1$. Therefore \[\lim_{i \to +\infty} D_p^i(x) = \lim_{i \to +\infty} \alpha^ix_1 = -\infty.\qedhere\]
\end{proof}

Theorem \ref{theorem:main1} gives the form of the $\inc$ sequence of any integer $x$ such that $\ord(x) \geq p$. It is a sequence of nonnegative integers that rapidly decreases until it is less than $p$, after which it is constant, where the first term determines all subsequent terms.  This sequence can be interrupted by strings of $-1$'s, where the length of the string is equal to the  previous nonnegative integer modulo $p$.    
  
The question remains about whether all sequences of the form 
\[(\underbrace{-1, -1, \ldots, -1}_{(\ell \bmod p) \text{ copies}}, \mathcal{S}_{\ell_1}, \mathcal{S}_{\ell_2}, \mathcal{S}_{\ell_3}, \ldots, \mathcal{S}_{\ell_N}, \mathcal{S}_{\ell_N}, \mathcal{S}_{\ell_N}, \ldots)\]
has an $x$ such that this is its $\inc$ sequence.  Since the initial $\ell_1$ determines all the subsequent $\ell_i$'s, the answer is no.  But we do have some control over the length of the interrupting $-1$'s and the number of jumps $N$.   

\begin{theorem} \label{theorem:main_reverse}
Let $N > 0$ be an integer. For any $N+1$ integers $0 \leq i_0, i_1, i_2, \ldots, i_N \leq p-1$, there exists a positive integer $\ell$ such that the $\inc$ sequence of $x = p^{\ell}$ is of the form
\[(\underbrace{-1, -1, \ldots, -1}_{i_0 \text{ copies}}, \mathcal{S}_{\ell_1}, \mathcal{S}_{\ell_2}, \mathcal{S}_{\ell_3}, \ldots, \mathcal{S}_{\ell_N}, \mathcal{S}_{\ell_N}, \mathcal{S}_{\ell_N}, \ldots),\]
such that for all $1 \leq j \leq N$, $\mathcal{S}_{\ell_j}$ is the $\ell_j$-segment with $i_j$ consecutive copies of $-1$.
\end{theorem}
\begin{proof}
Given any sequence of $N+1$ nonnegative integers $0 \leq i_0, i_1, i_2, \dots, i_N \leq p-1$, we construct a sequence $k_N, k_{N-1}, \dots, k_1, k_0$ as follows. Let $k_N = i_N + 1$ and $k_j = p^{k_{j+1}} + i_j + 1$ for $0 \leq j \leq N-1$. Let $\ell := k_0-1$ and $x = p^{k_0-1}$. We first show that $\ell_j = k_j$ for all $1 \leq j \leq N$. By definition of $\ell_1$, we can compute \[\ell_1 = \ord(\pfloor{\ell}) = \ord(\pfloor{k_0-1}) = \ord(\pfloor{p^{k_1}+i_0}) = \ord(p^{k_1}) = k_1.\]
Suppose for $2 \leq j \leq m-1 < N$, we have $\ell_j = k_j$. Then we have
\[\ell_m = \ord(\pfloor{\ell_{m-1}-1}) = \ord(\pfloor{k_{m-1}-1}) = \ord(\pfloor{p^{k_m}+i_{m-1}}) = \ord(p^{k_m}) = k_m.\]
It remains to show that the $\inc$ sequence of $x$ is of the desired form. As $k_0 -1 = p^{k_1}+i_0$, it is clear that the first $i_0$ terms of the $\inc$ sequence of $x$ are $-1$ and $\ord(D^{i_0}_p(x))=p^{k_1}$. Suppose $D_p^{i_0}(x) = \alpha p^{p^{k_1}}$. Set $k_1':= (k_1-1 \bmod p)+1 = (i_1 \bmod p) + 1 = i_1 + 1$. By Lemma \ref{lemma:basic2}, we know that the next $k_1'$ terms of the $\inc$ sequence is a $k_1$-segment $\mathcal{S}_{k_1} = \mathcal{S}_{\ell_1}$. The number of consecutive $-1$ in this $k_1$-segment $\mathcal{S}_{k_1}$ is equal to $k_1'-1=i_1$ as desired. The proof is complete by induction.
\end{proof}

For any nonzero integer $x = ap^{\ell}$ with $p \nmid a$ and $\ell \geq 1$, Theorem \ref{theorem:main1} implies that the $\ord$ sequence increases (not necessarily consecutively) exactly $N-1$ times before it reaches the periodic cycle. Therefore it is impossible for the $\ord$-sequence to continue to increase indefinitely.  However, we see that the number of jumps can be made arbitrarily large.  And as the following corollary tells us, we can even prescribe these jumps to be consecutive. 

\begin{corollary}
For any integer $N \geq 1$, there exists $x \in \mathbb{Z}$ such that \[\ord(x) < \ord(D_p(x)) < \cdots < \ord(D_p^N(x)).\]  
\end{corollary}
\begin{proof}
Use Theorem \ref{theorem:main_reverse} by choosing $i_0 = i_1 = \cdots = i_{N} = 0$.
\end{proof}

\section{Anti-partial derivatives}

Not every integer has an anti-partial derivative with respect to $p$. We would like to have a criterion to determine when an integer does have an anti-partial derivative (Question 1). If we know an integer has an anti-partial derivative, a natural extension of Question 1 is how many anti-partial derivatives can it have (Question 3)? In order to answer that, it would be useful to know when two integers have the same partial derivatives (Question 2). Furthermore, if we know that there exists at least one integer that has exactly $n$ anti-partial derivatives, then we can also ask how many integers there are that have exactly $n$ anti-partial derivatives (Question 4). In this section, we will attempt to answer these four questions:
\begin{enumerate}[1.]
    \item When does an integer have an anti-partial derivative?
    \item When do two integers have the same partial derivatives?
    \item How many anti-partial derivatives can an integer have?
    \item For any positive integer $n$, how many integers that have exactly $n$ anti-partial derivatives?
\end{enumerate}
For any nonzero integer $x$, if $\ord(x) > 0$, then there exist unique integers $a \neq 0$, $b > 0$, and $k \geq 0$ with $p \nmid ab$ such that $x = ap^{bp^k}$. We call this the \emph{standard form} of $x$.

\subsection{When does an integer have an anti-partial derivative?}

We start this subsection with an example of a family of integers that does not have an anti-partial derivative.

\begin{example} \label{example:no_antipd}
Suppose $y = a_0p^{p-1}$ for some integer $a_0$ where $p \nmid a_0$. If $D_p(x) = y$, then it is clear that $x \neq 0$ and $\ord(x) > 0$. Let $x = ap^{bp^k}$ be in standard form such that $D_p(x) = a_0p^{p-1}$. Then $abp^{bp^k+k-1} = a_0p^{p-1}$. This implies that $bp^k+k = p$. As $b \geq 1$, we know that $k = 0$. Therefore $b = p$ and this contradicts to $p \nmid ab$. Therefore $a_0p^{p-1}$ does not have an (integral) anti-partial derivative (with respect to $p$).
\end{example}

The main goal of this subsection is to determine when an integer $y$ has an anti-partial derivative. We first consider two simple cases.
\begin{itemize}
    \item If $y=0$, then $D_p(a) = 0$ for all $a \in \mathbb{Z}$ with $p \nmid a$.
    \item If $\ord(y) = 0$, then $D_p(x) = y$ if and only if $x = yp$.
\end{itemize}
Now we assume that $\ord(y) > 0$. If $D_p(x) = y$, then $\ord(x) > 0$. 

\begin{lemma} \label{lemma:antider_basic}
     Let $y = a_0p^{\ell_0}$ where $p \nmid a_0$ and $\ell_0 >0$. For any $x = ap^{bp^k} \in \mathbb{Z}$ in standard form, $D_p(x) = y$ if and only if
     \begin{align}
         ab &= a_0, \label{eq:bk1}\\
        bp^k+k-1 &= \ell_0. \label{eq:bk2}
    \end{align}
\end{lemma}
\begin{proof}
Since $D_p(x) = abp^{bp^k+k-1}$, it is clear that $abp^{bp^k+k-1} = a_0p^{\ell_0}$ if and only if \eqref{eq:bk1} and \eqref{eq:bk2} hold.
\end{proof}

Using Lemma \ref{lemma:antider_basic}, we can show that any nonzero integer has finitely many anti-partial derivatives. This is already known as a result of \cite[Theorem 1]{HMT1}.

\begin{corollary} \label{corollary:finite_partial}
For each nonzero $y \in \mathbb{Z}$, $\{x \in \mathbb{Z} : D_p(x) = y\}$ is finite (possibly empty).
\end{corollary}
\begin{proof}
If $\ord(y) = 0$, then $D_p(x) = y$ if and only if $x = yp$. If $\ord(y) > 0$, then \eqref{eq:bk2} implies that $k \leq \log_p(\ell_0)$. Since $b$ is determined by $k$ by \eqref{eq:bk2} and $a$ is determined by $b$ by \eqref{eq:bk1}, there are only finitely many $a,b,k$ such that \eqref{eq:bk1} and \eqref{eq:bk2} hold.
\end{proof}

\begin{corollary} \label{coro:k_linear}
For any two nonzero integers $x_1$ and $x_2$ with $\ord(x_1), \ord(x_2) > 0$, let $x_1 = a_1p^{b_1p^{k_1}}$ and $x_2 = a_2p^{b_2p^{k_2}}$ be in standard form such that $D_p(x_1) = D_p(x_2)$. The following three statements are equivalent.
\[ \text{\normalfont (i) } k_1 < k_2 \qquad \text{\normalfont (ii) } b_1p^{k_1} > b_2p^{k_2} \qquad \text{\normalfont (iii) } b_1 > b_2  \]
If $D_p(x_1)=D_p(x_2)>0$, then the above three statements are further equivalent to the following two statements.
\[ \text{\normalfont (iv) } a_1 < a_2 \qquad \text{\normalfont (v) } x_1 < x_2\]
\end{corollary}
\begin{proof}
If $D_p(x_1) = D_p(x_2)$, then we have 
\begin{align}
     a_1b_1 &= a_2b_2  \label{eq:bk3}\\
    b_1p^{k_1}+k_1-1 &= b_2p^{k_2}+k_2-1, \label{eq:bk4}
\end{align}
\eqref{eq:bk4} implies that (i) and (ii) are equivalent. If (i) and (ii) hold, then (iii) holds. If (iii) holds, then \eqref{eq:bk4} implies that (i) holds. 

Now suppose $D_p(x_1)=D_p(x_2)>0$. Hence $a_1, a_2, x_1, x_2 > 0$. \eqref{eq:bk3} implies that (iii) and (iv) are equivalent. It remains to show that (v) is equivalent to (ii). Note that
\[\dfrac{x_1}{x_2} = \dfrac{a_1}{a_2} \cdot \dfrac{p^{b_1p^{k_1}}}{p^{b_2p^{k_2}}} = \dfrac{b_2}{b_1} p^{b_1p^{k_1}-b_2p^{k_2}} = \dfrac{b_2}{b_1} p^{k_2-k_1} = \dfrac{b_2p^{k_2}}{b_1p^{k_1}}.\]
Therefore (v) is equivalent to (ii).
\end{proof}

If $D_p(x_1)=D_p(x_2) < 0$, then (iv) and (v) of Corollary \ref{coro:k_linear} need to be replaced by $a_1 > a_2$ and $x_1 > x_2$ respectively.

\subsection{When do two integers have the same partial derivatives?}

Anti-partial derivatives do not exist uniquely. For example
\[D_p(p^{p^{p+1}+p}) = D_p( (p^p+1) p^{p^{p+1}}  ).\]
The main goal of this subsection is to determine when two integers have the same partial derivatives. If $D_p(x) = 0$, then $x = 0$ or $\ord(x) = 0$. This means that the set of anti-partial derivatives of 0 is 
\[\{x \in \mathbb{Z} \, : \, p \nmid x\} \, \cup \, \{0\}.\]
Now we suppose that $\ord(x) > 0$ and let $y = D_p(x)$. By Corollary \ref{coro:k_linear}, there is an element $x_0$ in the set of all anti-partial derivative of $y$ with the smallest value $k$. We call $x_0$ the \emph{primitive} anti-partial derivative of $y$ and denote $x_0 = a_0p^{b_0p^{k_0}}$ in the standard form.

\begin{theorem} \label{theorem:primitive}
Let $x_0 = a_0p^{b_0p^{k_0}} \in \mathbb{Z}$ be the standard form of the primitive anti-partial derivative of $y = D_p(x_0)$. Define
\[C = C(x_0) := \Big\{c \in [0, b_0) \cap \mathbb{Z} \, :  \, p^{ p^{k_0} c } \midd (b_0-c) , \;  \frac{b_0-c}{p^{p^{k_0}c}} \mid a_0b_0 \Big\}.\]
There is a one-to-one correspondence between $C$ and the set of all anti-partial derivatives of $y$. Furthermore, if $x = ap^{bp^k}$ (in its standard form) is an anti-partial derivative of $y$, then there exists a unique $c \in C$ such that
\[k = p^{k_0}c+k_0, \quad b = \frac{b_0-c}{p^{k-k_0}} = \frac{b_0-c}{p^{p^{k_0}c}} \in \mathbb{Z}, \quad a = \frac{a_0b_0}{b} \in \mathbb{Z}.\] 
\end{theorem}
\begin{proof}
We show that any anti-partial derivative $x=ap^{bp^k}$ of $y$ is associated with a unique $c \in C$. If $x = x_0$, then we associated $x$ with $c=0$. If $x \neq x_0$, since $x_0$ is the primitive anti-partial derivative, we know that $k > k_0$. Then $p^{k_0}(b_0 - bp^{k-k_0}) = k-k_0$. Since $p \nmid bb_0$, we know that $p \nmid (b_0-bp^{k-k_0})$. Therefore $p^{k_0} \midd (k-k_0)$. Let $k-k_0 = p^{k_0}c$ where $c > 0$ and $p \nmid c$. By plugging $k-k_0=p^{k_0}c$ into $p^{k_0}(b_0 - bp^{k-k_0}) = k-k_0$, we get $b_0 - bp^{k-k_0} = c$. Since $bp^{k-k_0} > 0$, we know that $c<b_0$. Since $p \nmid b$, we know that $p^{k-k_0} \midd (b_0 - c)$. Since $a \in \mathbb{Z}$, we also have $b \mid a_0b_0$, that is, $\frac{b_0-c}{p^{p^{k_0}c}} \mid a_0b_0$.

Then we show that for each $c \in C$, we can define a unique $x=ap^{bp^k}$ such that $D_p(x) = y$. Since $c < b_0$ and $p^{p^{k_0} c} \midd (b_0-c) $, we know that $b > 0$ is an integer and $\text{ord}_p(b) = 0$. Since $b \mid a_0b_0$, we know that $a \in \mathbb{Z}$ and $\text{ord}_p(a) = 0$ as well. We can compute
\[bp^k+k-1 = \frac{b_0-c}{p^{k-k_0}} p^k + k -1 = (b_0-c)p^{k_0} + p^{k_0}c+k_0-1 = b_0p^{k_0} + k_0 -1.\]
Therefore
\[D_p(x) = abp^{bp^k+k-1} = a_0b_0p^{b_0p^{k_0}+k_0-1} = D_p(x_0) = y. \qedhere\]
\end{proof}

\subsection{How many anti-partial derivatives can an integer have?}

By Corollary \ref{corollary:finite_partial}, we know that every nonzero integer can only have finitely many anti-partial derivatives. This partially answers Question 3 but we can further ask for any given positive integer $n$, is there an integer $x$ with exactly $n$ anti-partial derivatives? The answer is yes. In order to prove that, we will construct $k_0$ such that $x_0$ is the primitive anti-partial derivative of $D_p(x_0)$ and then construct $a_0$ and $b_0$ so that $D_p(x_0)$ has exactly $n$ anti-partial derivatives. We first show that for certain values of $k_0$, $x_0$ is always the primitive anti-partial derivative of $D_p(x_0)$ no matter how $a_0$ and $b_0$ are defined as long as $p \nmid a_0b_0$.

\begin{lemma} \label{lemma:define_k}
   Fix an integer $m \geq 2$ and let $k_0 := p + p^2 + p^3 + \cdots + p^m$. For any integers $a_0, b_0$ with $b_0 > 0$ such that $p \nmid a_0b_0$, let $x_0 := a_0p^{b_0p^{k_0}}$. Then $x_0$ is the primitive anti-partial derivative of $D_p(x_0)$.
\end{lemma}
\begin{proof}
Suppose $x_0$ is not the primitive anti-partial derivative of $D_p(x_0)$. Let $x = ap^{bp^k}$ be an anti-partial derivative of $D_p(x_0)$ such that $0\leq k<k_0$. Since $D_p(x) = D_p(x_0)$, we get $bp^k+k=b_0p^{k_0}+k_0$. This means that $p^k \midd k_0 - k$. It is clear that $k \neq 0$ because $p \midd k_0 - 0$. It is also clear that $k \neq 1$ because $p \nmid k_0 -1$. Suppose $k \geq 2$. If $p^k \midd k_0 - \ell$ for some $\ell \in \mathbb{Z}$, then $\ell \geq p + p^2 + \cdots p^{k-1} > k$. Therefore there does not exist an anti-partial derivative $x=ap^{bp^k}$ with $k<k_0$. This means that $x_0$ is the primitive anti-partial derivative of $D_p(x_0)$.
\end{proof}

We now construct $b_0$ such that there are exactly $n-1$ different possible values of $c \in (0,b_0) \cap \mathbb{Z}$ such that $p^{p^{k_0}c} \midd (b_0-c)$ for any $k_0 \geq 0$. This means that the set $C(x_0)$ (as defined in Theorem \ref{theorem:primitive}) has at most $n$ elements (including $0$).

\begin{lemma} \label{lemma:define_b}
Fix integers $n > 0$ and $k_0 \geq 0$. Let $c_1 = 0$. For $2 \leq i \leq n+1$, let $c_i := p^{p^{k_0}c_{i-1}} + c_{i-1}$ and $b_0 := c_{n+1}$. We have
\[\{c \in (0, b_0) \cap \mathbb{Z} \; : \; p^{p^{k_0}c} \midd b_0-c \} = \{c_2, \ldots, c_n\}. \]
\end{lemma}
\begin{proof}
First, we prove $\supseteq$. By the definition of $c_j$, we know that $c_{j+1} > c_j$ and $p^{p^{k_0}c_{j}} \midd c_{j+1} - c_{j}$ for all $1 \leq j \leq n$. For $1 \leq i \leq n$, we have \[b_0 - c_i = \sum_{j=i}^{n} (c_{j+1} - c_{j}) = \sum_{j=i}^{n} p^{p^{k_0}c_j}.\] Therefore $p^{p^{k_0}c_i} \midd b_0-c_i$.

Second, we prove $\subseteq$. Let $c \in (0, b_0)$ be an integer such that $p^{p^{k_0}c} \midd b_0 - c$. If $c \not\in \{c_2, \dots, c_{n}\}$, then there exists a unique $3 \leq i \leq n+1$ such that $c_{i-1} < c < c_i$. As 
\[b_0 - c = \sum_{j=i}^{n} (c_{j+1} - c_{j}) + (c_i - c) = \sum_{j=i}^{n} p^{p^{k_0}c_j} + (c_i - c).\]
Since $c < c_j$ for all $i \leq j \leq n$, we have $p^{p^{k_0}c} \mid p^{p^{k_0}c_j}$. Since $p^{p^{k_0}c} \midd b_0 - c$, we know that $p^{p^{k_0}c} \midd c_i - c$. Let $\ell > 0$ with $p \nmid \ell$ such that $c_i = \ell p^{p^{k_0}c} + c$. By definition of $c_i$, we know that $c_i = p^{p^{k_0}c_{i-1}} + c_{i-1}$. This means that \[\ell p^{p^{k_0}c} + c = p^{p^{k_0}c_{i-1}} + c_{i-1}.\]
Since $c_{i-1} < c$, we know that $\ell p^{p^{k_0}c} < p^{p^{k_0}c_{i-1}}$. But this is not possible as $\ell \geq 1$ and $c> c_{i-1}$. Hence there does not exist an integer $c \in (0, b_0)$ such that  $p^{p^{k_0}c} \midd b_0 - c$ and $c \notin \{c_2, \dots, c_{n}\}$. This completes the proof of $\subseteq$.
\end{proof}

In order for $D_p(x_0)$ to have exactly $n$ anti-partial derivatives, we want to construct $a_0$ so that \[\frac{b_0-c}{p^{p^{k_0}c}} \mid a_0b_0\] for every $0 < c < b_0$ satisfying $p^{p^{k_0}c} \midd b_0 - c$. This can be achieved by the next lemma.

\begin{lemma} \label{lemma:define_a}
Fix integers $n > 0$ and $k_0 \geq 0$. Let $c_1, c_2, \dots, c_{n+1} = b_0$ be defined as in Lemma \ref{lemma:define_b}. For each $2 \leq i \leq n$, let $a_i := (b_0 - c_i) / p^{p^{k_0}c_i}$ and let $a_0 := \prod_{j=2}^n a_j$. Then $p \nmid a_0$ and the set $C(x_0)$ (as defined in Theorem \ref{theorem:primitive}) has exactly $n$ elements.
\end{lemma}
\begin{proof}
Lemma \ref{lemma:define_b} shows that $C(x_0) \subseteq \{c_1, c_2, \dots, c_n\}$. For each $2 \leq i \leq n$, we know that $(b_0 - c_i) / p^{p^{k_0}c_i} = a_i$ divides $a_0$ by the definition of $a_0$. Hence it divides $a_0b_0$. Since $p \nmid a_i$ for $2 \leq i \leq n$, we know that $p \nmid a_0$. Therefore $C(x_0) = \{c_1, c_2, \dots, c_n\}$.
\end{proof}

Now we are ready to give a fuller answer to Question 3 as well as an answer to Question 4 at the same time.

\begin{theorem} \label{theorem:infinite_derivative}
For each positive integer $n$, there are infinitely many integers $x_0$ such that $D_p(x_0)$ has exactly $n$ anti-partial derivatives.
\end{theorem}

\begin{proof}
Lemma \ref{lemma:define_k} implies that there are infinitely many $k_0$ such that $x_0$ is the primitive anti-partial derivative of $D_p(x_0)$. Lemmas \ref{lemma:define_b} and \ref{lemma:define_a} imply that for any $k_0$, we can construct $a_0$ and $b_0$ so that $x_0 = a_0p^{b_0p^{k_0}}$ is the primitive anti-partial derivative of $D_p(x_0)$ and $D_p(x_0)$ has exactly $n$ anti-partial derivatives. Therefore, for each positive integer $n$, there exists infinitely many integers $x_0 = a_0p^{b_0p^{k_0}}$ such that $D_p(x_0)$ has exactly $n$ anti-partial derivatives.
\end{proof}

\begin{remark} \label{remark:rational}
For any integer $x_0$, let $\mathcal{D}_{\mathbb{Z},p}(x_0)$ (resp., $\mathcal{D}_{\mathbb{Q},p}(x_0)$) be the set of all integral (resp., rational) anti-partial derivatives of $D_p(x_0)$. It is clear that $\mathcal{D}_{\mathbb{Z}}(x_0) \subset \mathcal{D}_{\mathbb{Q}}(x_0)$. Therefore Theorem \ref{theorem:infinite_derivative} implies that \cite[Conjecture 27]{HMT1} is false.
\end{remark}

Theorem \ref{theorem:infinite_derivative} remains true even if we allow rational anti-partial derivatives. We give a brief explanation as follows. Let $x_0 := a_0p^{b_0p^{k_0}}$ be defined as in Lemma \ref{lemma:define_k}, so $x_0$ is the primitive anti-partial derivative of $D_p(x_0)$ (even among all the rational anti-partial derivatives) with $b_0 > 0$. Let $x \neq x_0$ be a rational anti-partial derivative of $D_p(x_0)$. We can write $x = ap^{bp^{k}}$ where $0 \neq a \in \mathbb{Q}$, $0 \neq b \in \mathbb{Z}$, $k \in \mathbb{Z}$, $k \geq 0$, and $\ord(ab) \neq 0$. Since $x_0$ is the primitive anti-partial derivative of $D_p(x_0)$, we have $k > k_0$. As $D_p(x) = D_p(x_0)$, we have $b_0p^{k_0}+k_0-1 = bp^{k}+k-1$, or $p^{k_0}(b_0-bp^{k-k_0}) = k - k_0$. If $b<0$, then $k-k_0 < -bp^{k-k_0} < b_0 - bp^{k-k_0} < p^{k_0}(b_0 - bp^{k-k_0})$. This is a contradiction, so $b>0$. This is saying that if $x_0$ is the primitive rational anti-partial derivative of $D_p(x_0)$ with $\ord(x_0) > 0$, then any other rational anti-partial derivative $x$ of $D_p(x_0)$ also satisfies $\ord(x) > 0$. In this case, there is a one-to-one correspondence between the set of all rational anti-partial derivatives of $D_p(x_0)$ and the following set
\[C_{\mathbb{Q}}(x_0) := \Big\{c \in [0, b_0) \cap \mathbb{Z} \, :  \, p^{ p^{k_0} c } \midd (b_0-c)\Big\}.\]
If we compare $C_{\mathbb{Q}}(x_0)$ with $C(x_0)$ defined in Theorem \ref{theorem:primitive}, we have dropped the extra condition $\frac{b_0-c}{p^{p^{k_0}c}} \mid a_0b_0$. This is because we no longer need $a = a_0b_0/b$ to be an integer. Therefore Lemma \ref{lemma:define_b} still applies, that is, we can find $b_0$ such that $C_{\mathbb{Q}}(x_0)$ has exactly $n$ elements for any positive integer $n$. Hence, for any positive integer $n$, there are infinitely many integers $x_0$ such that $D_p(x_0)$ has exactly $n$ \emph{rational} anti-partial derivatives.

\section{Acknowledgments}
The authors would like to thank the  anonymous referee for carefully reading  the manuscript and the editor-in-chief for many useful comments and suggestions.

\bigskip
\hrule
\bigskip

\noindent 2010 {\it Mathematics Subject Classification}: Primary 11A25; Secondary 11A41, 11Y55.

\medskip

\noindent \emph{Keywords: arithmetic derivative, arithmetic partial derivative.}

\bigskip
\hrule
\bigskip

\noindent (Concerned with sequence \seqnum{A003415}.)

\end{document}